\documentclass[12, twoside]{amsart}
\usepackage{amsfonts,amssymb,amscd,amsmath,enumerate,verbatim,calc}
\usepackage[all]{xy}
\usepackage{color}

\newcommand{\m}{\mathfrak{m} }

\newcommand{\ann}{\operatorname{ann}}

\newcommand{\embdim}{\operatorname{embdim}}
\newcommand{\cha}{\operatorname{char}}
\newcommand{\spa}{\operatorname{Span}}

\newcommand{\soc}{\operatorname{Soc}}

\newcommand{\rank}{\operatorname{rank}}

\newcommand{\rad}{\operatorname{rad}}

\theoremstyle{plain}

\newtheorem{theorem}{Theorem}[section]
\newtheorem{corollary}[theorem]{Corollary}
\newtheorem{lemma}[theorem]{Lemma}
\newtheorem{proposition}[theorem]{Proposition}
\newtheorem{question}[theorem]{Question}

\theoremstyle{definition}
\newtheorem{definition}[theorem]{Definition}

\newtheorem{remark}[theorem]{Remark}

\newtheorem{observation}[theorem]{Observation}

\newtheorem{notations}[theorem]{Notations}
\theoremstyle{remark}

\makeatletter
\@namedef{subjclassname@2020}{%
	\textup{2020} Mathematics Subject Classification}
\makeatother
\AtBeginDocument{%
	\def\MR#1{}
}
\begin{document}
	
	\title[Compressed   zero-divisor graphs of Artinian rings ]{On the problem of the finiteness of the clique number of compressed   zero-divisor graphs of Artinian rings }
	\author{Ganesh~S.~Kadu}
	\address{Department of Mathematics, Savitribai Phule Pune University, Pune 411 007, India}
	\email{ganeshkadu@gmail.com, gskadu@unipune.ac.in}
	\thanks{The author thanks  DST-SERB for financial assistance under the project ECR/2017/00790.}
	\date{\today}
	


\subjclass[2020]{Primary 13A70, 05E40; Secondary 05C69, 13E39.}
	
	\begin{abstract} Let $R$ be an Artinian ring and let $\Gamma_E(R)$ be the compressed zero-divisor graph associated to $R.$ The question of when the clique number  $\omega(\Gamma_E(R)) < \infty$ was raised by J. Coykendall, S. Sather-Wagstaff, L. Sheppardson, and S. Spiroff, see \cite[Section 5]{Coy}. They proved that if  $\ell(R) \leq 4 $ then  $\omega(\Gamma_E(R)) < \infty.$ When $\ell(R)=6$, they gave an example of a local ring $R$ where  $\omega(\Gamma_E(R)) = \infty$ is possible by using the trivial extension of an Artinian local ring by its dualizing module. The question of what happens when $\ell(R) =5$ was stated as an open problem. We show that if $\ell(R)=5$  then $\omega(\Gamma_E(R)) < \infty.$ We first reduce the problem to the case of a local ring $(R, \m, k).$ We then enumerate all possible Hilbert functions of $R$ and show that the $k$-vector space $\m/\m^2$ admits a symmetric bilinear form in some cases of the Hilbert function. This allows us to relate the orthogonality in the bilinear space $\m/\m^2$  with the structure of zero-divisors in $R.$ For instance, in the case when $\m^2$ is principal and $\m^3=0$, we show that $R$ is Gorenstein if and only if the symmetric bilinear form on $\m/\m^2$ is non-degenerate. Moreover, in the case when $\ell(R) =4,$  our techniques also yield a simpler and shorter proof of the finiteness of  $\omega(\Gamma_E(R))$ avoiding, for instance, Cohen structure theorem.
	\end{abstract}
	\maketitle
	\section{ Introduction}
	Let $R$ be a commutative ring with unity.  In \cite{IB}, I. Beck associated to a ring $R,$ a simple graph with the vertex set consisting of all elements of  $R$, and the two vertices  $a, b \in R$ are adjacent whenever $ab=0.$ This was later modified by D. F. Anderson and P. S. Livingston in \cite{AL} to include only the nonzero zero-divisors of $R$ as its vertex set and two vertices are adjacent if their product is zero. This graph is denoted by $\Gamma(R).$ These graphs are an interesting class of graphs, and their study offers a rich interplay between the ring theoretic properties of $R$ and the graph theoretic properties of $\Gamma(R),$ as evidenced by Beck's paper \cite{IB} and the series of papers following it, see  \cite{AN},  \cite{AL1}, \cite{AL2}, \cite{AL}, \cite{Coy}.  The zero-divisor graphs possess good symmetry properties as their graph automorphism groups tend to be quite large in general, see \cite[Theorem 3.2]{AL} and \cite[3.9]{SM}. Furthermore, the zero-divisor graphs often also have the property that their chromatic number  $\chi(\Gamma(R))$ and the clique number $\omega(\Gamma(R))$ are equal,  see for example \cite[Theorems 6.10 and 6.11]{IB}, \cite[Corollary 3.3]{AN}.\\  
	\indent  A compressed version of $\Gamma(R)$ was introduced and studied by S. Mulay in \cite{SM} and is denoted by $\Gamma_E(R).$ This graph $\Gamma_E(R)$, known as the \emph{compressed zero-divisor graph} of a ring $R,$   is obtained from $\Gamma(R)$ by letting the two vertices $a, b \in V(\Gamma(R))$ as equivalent if $\ann(a) =\ann(b).$ So the vertex set of $\Gamma_E(R)$ is the set of equivalence classes of vertices of $\Gamma(R)$ and the adjacency is inherited from that in $\Gamma(R),$ i.e., an equivalence class $[a]$ is adjacent to $[b]$ if $ab=0$ or equivalently $[a][b] =0.$ Observe that $\ann(a)$ is essentially the neighbourhood of $a$ in $\Gamma(R).$ Hence this equivalence relation compresses $\Gamma(R)$ by identifying all the vertices in each equivalence class, i.e., vertices having the same neighborhood in $\Gamma(R)$ are identified in $\Gamma_E(R).$ We, therefore, have a surjective graph map,
	\[ \psi : V(\Gamma(R)) \longrightarrow V(\Gamma_E(R)) \]
	given by $\psi(a) = [a].$ The map $\psi$ has the property that if $a$ and $b$ are adjacent in $\Gamma(R),$ then  $\psi(a)$ is adjacent to $\psi(b)$ or $\psi(a) =\psi(b).$ It is an interesting problem to study $\Gamma(R)$ using the more simplified  $\Gamma_E(R).$ 
	For instance, the graph automorphism group of $\Gamma(R)$ can be determined from the automorphism group of $\Gamma_E(R),$ see \cite[3.9]{SM}.  For this reason, it is often helpful to look at the more simplified compressed graph $\Gamma_E(R)$ than the usual $\Gamma(R).$ The graph $\Gamma_E(R)$ was further studied by D.F. Anderson and J. LaGrange \cite{AL1} (also see \cite{AL2}), S. Spiroff and C. Wickham \cite{SW}, M. Axtell, N. Baeth, and J. Stickles \cite{ABS}. For a detailed and nice survey of the zero-divisor graphs with a special emphasis to $\Gamma_E(R),$ see J. Coykendall et al. \cite{Coy}.
	
	Let $\ell(R)$ denote the length of ring $R$, i.e., the largest length of any of its chains of ideals.  From now on, we assume that $R$ is an Artinian ring, i.e., it has a finite length. The question of when the clique number of $\Gamma_E(R))$ is finite, i.e., $\omega(\Gamma_E(R)) < \infty$ was posed by J. Coykendall, S. Sather-Wagstaff, L. Sheppardson, and S. Spiroff in \cite[Section 5]{Coy}.  They proved that if $\ell(R) \leq 3, $ then  $\Gamma_E(R)$ is a finite graph and hence  $\omega(\Gamma_E(R)) < \infty.$ When $\ell(R)=4,$ they proved that $\omega(\Gamma_E(R)) < \infty,$ noting however that $\Gamma_E(R)$ can be an infinite graph. To do this, they reduce to the local case and work with the explicit presentations of Artinian local rings as quotients of power series rings by using the Cohen structure theorem. When $\ell(R)=6,$ they give an example of a ring $R,$ where  $\omega(\Gamma_E(R)) = \infty$ is possible. To construct such an example in length $6,$ they use the trivial extension of an Artinian local ring $R$ by its dualizing module, viz. the injective hull $E_R(k)$ of its residue field $k.$ 
	The question of what happens when $\ell(R) =5$ was stated as an open problem, J. Coykendall et al. see \cite[Question 5.9]{Coy}. More precisely,\vspace{0.1 cm} \newline 
	{\bf Problem: } If $R$ is Artinian ring with $\ell(R) =5,$ then is $\omega(\Gamma_E(R)) < \infty?$
	
	\vspace{0.1 cm} 
	
	In this paper we resolve this problem and show that if $R$ is an  Artinian ring with $\ell(R)=5,$  then $\omega(\Gamma_E(R)) < \infty.$ In the local case, we, in fact, show that $\omega(\Gamma_E(R)) \leq 5- \dim_k \soc R.$ This shows, in particular, that the example given by J. Coykendall et al. of length $6$ Artinian ring where $\omega(\Gamma_E(R)) = \infty$ is indeed a minimal one. Moreover, in the case when $\ell(R) =4,$ our methods also yield a short and easy proof of   $\omega(\Gamma_E(R)) < \infty.$ 
	
	We first reduce the problem to the case of a local ring $(R, \m, k)$ and then work by enumerating all possible Hilbert functions in the length five case. We define a symmetric bilinear form $\phi$ on the $k$-vector space $\m/\m^2$ in some of the cases of the Hilbert function of $R.$ There are two results that are key to our analysis of the graph $\Gamma_E(R),$ namely, Proposition \ref{bilinear} and Lemma \ref{socle}(ii). The first of these, Proposition \ref{bilinear}, allows us to relate zero-divisors in $R$ with the orthogonality of vectors in the bilinear space $(\m/\m^2, \phi)$ in the case when $\m^2$ is principal and $\m^3=0.$  We also observe that the radical $(\m/\m^2)^\perp $ of the bilinear space  $(\m/\m^2, \phi)$ can be completely described in terms of $\soc R.$  It is interesting to note, in this case, that this symmetric bilinear form $\phi$ is non-degenerate if and only if $R$ is Gorenstein local, see Proposition \ref{bilinear}(iii). The second of these results, Lemma \ref{socle}(ii), is an observation that enables us to map distinct vertices in $\Gamma_E(R)$ to distinct one-dimensional subspaces in $\m/\m^2.$
	As the first application of  Lemma \ref{socle}(ii), we give an easier and shorter proof of the finiteness of  $\omega(\Gamma_E(R))$ in the case when $\ell(R) =4.$ Note that the original proof (see \cite[Proposition 5.8(v)(c)]{Coy}) of the case $\ell(R)=4,$ is an intricate case by case analysis and uses Cohen's structure theory, in some specific cases, to obtain presentations of the rings as quotients of power series rings. \vspace{0.05 cm} \newline  
	\indent Another important ingredient of our investigation is the use of symmetric bilinear forms on finite-dimensional vector spaces.  In Section $2$ on bilinear forms, we recall the standard notions from the orthogonal geometry of symmetric bilinear spaces over arbitrary fields. We also prove a technical result that, in essence, forbids a vector space of dimension $3$ from having an orthogonal set of size bigger than $3.$ To our surprise, the proof of this result is a little longer than what one might expect generally. We prove this result as Lemma \ref{ortho} and is crucially used in Proposition \ref{bound}. It is worth noting that one can have an orthogonal subset in bilinear space $V$ of size greater than $\dim V$ even if the symmetric bilinear form on $V$ is non-degenerate. This is pointed out in the discussion that precedes Lemma \ref{ortho}. We expect the general version of Lemma \ref{ortho} to hold for higher dimensional bilinear spaces and we state the higher dimensional formulation as Question \ref{question}. \vspace{0.05 cm} \newline  
	\indent  We then set out to prove the finiteness of $\omega (\Gamma_E(R)).$  This is done by first enumerating all possible Hilbert functions of a length five local ring and treating each case separately in the Propositions \ref{bound}, \ref{bound1}, and \ref{bound2}. One of our main ideas is to start with a clique in $\Gamma_E(R)$  and map its vertices to one-dimensional subspaces in $\m/\m^2.$ Since adjacency of the vertices in $\Gamma_E(R)$ corresponds to the orthogonality of the vectors in the bilinear space $\m/\m^2,$ we see that a clique in $\Gamma_E(R)$ gives rise to an orthogonal set in $\m/\m^2.$ We then use classical orthogonal geometry of symmetric bilinear spaces to show that the cliques of large sizes are not possible. The most interesting case is Proposition \ref{bound}, where we make use of the fact that $\soc R $ is one-dimensional (Gorenstein local ring) is equivalent to the non-degeneracy of the bilinear form $\phi.$ In the case when  $\dim_k \soc R >1$  we go modulo the $\rad(\m/\m^2),$  which is determined by $\soc R,$ to reduce to the non-degenerate case. We finally put all the results together to prove the main result, viz. Theorem \ref{main}.  
	\begin{notations}
		We mention some notations and abbreviations used throughout the paper. Let $\mathbb N$ denote the set of natural numbers and $\mathbb N_{ \geq 0} = \mathbb N \cup \{0\}.$  If $V$ is a vector space over field $k$ and $ v \in V,$ then a subspace generated by $v$ is denoted by $\langle v \rangle .$ The dimension of vector space $V$ over field $k$ is denoted by $\dim_k V.$ If $R$ is a ring and $ a \in R,$ then a principal ideal generated by $a$ is denoted by $(a).$ Abbreviation PIR stands for a principal ideal ring. The length of a ring $R$ is denoted by $\ell(R).$ We let $ \omega(G)$ denote the clique number of a graph $G,$, i.e., the size of a clique of the largest size in $G.$ 
	\end{notations}
	
	\section{bilinear forms}
	
	In this section, we recall some of the basic notions and results that we need from the theory of symmetric bilinear forms on finite dimensional vector spaces. We also prove a technical result relating to orthogonal geometry of bilinear spaces that will be used later in Section $3.$ \newline 
	\indent Let $V$ be a finite dimensional vector space over a field $k.$ Let $\dim V=n.$ A \textit{bilinear form}   on $V$ is a map $\phi : V \times V \rightarrow k$ that is linear in both the variables. A \textit{symmetric bilinear form} on $V$ is a bilinear form $\phi$ on $V$ that satisfies the condition $\phi(v, w) =\phi(w,v)$ for all $v, w \in V.$ Suppose now that vector space $V$ admits a symmetric bilinear form $\phi.$ We say that vectors $v$ and $w$ in $V$ are \textit{orthogonal} w.r.t. $\phi$ if $\phi(v, w) =0.$ A vector $v \in V$ is called \textit{self-orthogonal} (or \textit{isotropic}) w.r.t. $\phi$ if $\phi(v, v) =0.$ Given a subspace $W$ of $V,$ we have the orthogonal complement $W^\perp$ of $W$ in $V$  given by, $$W^\perp= \{ v \in V \mid \phi(v, w) =0, ~ \forall ~ w \in W\}.$$ \textit{Radical} of $V,$ denoted by $V^\perp,$ is the orthogonal complement of the whole space $V.$  Radical of $V$ is also sometimes known as the null space of $V$ or the degeneration space of $V$ etc. Symmetric bilinear form $\phi$ is called \textit{non-degenerate} if $V^\perp =0.$ If $\mathcal B= \{ v_1, v_2, \ldots, v_n\}$ is a basis of $V,$ then we have  $n \times n$ matrix $ A=[\phi(v_i, v_j)]$ representing the bilinear form $\phi.$ Note that $\phi$ is symmetric if and only if $A$ is symmetric.  \\\
	 \indent Some good references where these results can be found are \cite{EA}, \cite{LG}. An excellent reference is expository notes on Bilinear Forms by K. Conrad to be found at his homepage. In the following proposition, we collect together some basic facts on bilinear forms which we need later. It is easy to prove (i).  For the proof of (ii) see \cite[Theorem 4.2]{LG} and for (iii) see \cite[Theorem 11.7]{RS}. 
	\begin{theorem}\label{symbi}
		Let $V$ be a finite dimensional vector space over a field $k$ of dimension $n$ equipped with a symmetric bilinear form $\phi$. Then we have,
		
		\begin{enumerate}
			\item [\rm (i)] $\phi $ is non-degenerate if and only if matrix representation of $\phi$ in any basis $\mathcal B$ is non-singular.
			\item [\rm (ii)] If $\cha k \neq 2$ then there exists an orthogonal basis of $V$ corresponding to $\phi.$
			\item [\rm (iii)]  $\phi$ induces a non-degenerate symmetric bilinear form on $V/V^\perp.$ Moreover writing $ V = V^\perp \oplus W$ for a subspace $W$ complementary to $V^\perp,$ we get $W$ is non-degenerate space  w.r.t. restricted form $\phi|_W$ on $W.$ 		 
		\end{enumerate}
	\end{theorem}
	
	\indent Let $V$ be an $n$-dimensional  vector space  over a field $k,$ equipped with non-degenerate symmetric bilinear form $\phi.$ It is well known if $k=\mathbb R$ or $\mathbb C$ and $\phi$ is positive definite bilinear form (inner product) on $V,$  then $V$ can not have an orthogonal set of distinct vectors $\{v_1, v_2, \ldots, v_m \}$ with $m > n.$ It is interesting to note that, over arbitrary fields, this result fails even for vector spaces equipped with non-degenerate symmetric bilinear form $\phi$ on $V,$ i.e., $V$ can have more than $\dim V$ many distinct orthogonal vectors w.r.t. $\phi.$ This happens primarily due to the presence of self-orthogonal (isotropic) vectors in $V$ even when $\phi$ is non-degenerate. One way of constructing such orthogonal sets is to start with one self-orthogonal  vector $v \in V.$  For any $a_1, a_2, \ldots, a_m$  distinct $m$ scalars in the field $k,$  consider the set $ \{a_1v, a_2v, \ldots, a_mv \}.$ This is an orthogonal set in $V$ of size $m>n.$ \\
	\indent In the following result, we show that in dimension $3,$ in some sense, this is the only reason for the occurrence  of  orthogonal sets of size bigger than $\dim V.$ We expect this result to hold more generally in higher dimensions and we propose the higher dimensional formulation as  Question \ref{question}.  We prove the result  in the case of dimension $3.$ We note that this is the case in which it is crucially used in the Proposition \ref{bound}. To our surprise, the proof of this result is a little longer than what one might expect generally. 
	\begin{lemma}\label{ortho}
		Let $V$ be a  vector space over a field $k$ with $\dim V=3$ equipped with non-degenerate symmetric bilinear form $\phi.$  If $\{v_1, v_2, v_3, v_4 \}$ is an orthogonal set of non-zero vectors w.r.t. $\phi$ then, $v_i= \alpha v_j$ for some $ i \neq j$ and $\alpha \in k.$ 
	\end{lemma}
	\begin{proof} Let $W= \spa \{v_1, v_2, v_3, v_4 \}.$ Suppose first that  $W = V.$ Since $\dim V=3,$ there exists a set of three vectors among $v_1, v_2, v_3, v_4$ that forms a basis of $V,$ say $\{v_1, v_2, v_3\}.$ Hence $ v_4 = \alpha_1 v_1+ \alpha_2 v_2 + \alpha_3 v_3$ for some $ \alpha_1, \alpha_2, \alpha_3 \in k.$ Taking bilinear product with $v_4,$ we find that $\phi(v_4, v_4)=0.$ Thus $\phi(v_4, v_i) =0$ for $1 \leq i \leq 4,$ showing that $ v_4 \in V^{\perp}.$ Since $\phi$ is non-degenerate, we have  $V^{\perp}=0$ and hence $v_4=0.$ This is not possible, as all $v_i$'s are non-zero vectors. \vspace{0.1 cm} \newline 
		\indent So, assume now that $W$ is a proper subspace of $V.$ If $\dim W =1,$ then the conclusion is immediate. If $\dim W =2,$ then two of the four vectors $v_1, v_2, v_3, v_4$ forms a basis of  $W,$ say $\{v_1, v_2 \}.$ Thus $v_3 =\alpha_1 v_1+ \alpha_2 v_2$ and $v_4= \beta_1 v_1 + \beta_2 v_2.$ We use these expressions for $v_3$ and $v_4$ to compute $\phi(v_3, v_3)$ and $\phi(v_4, v_4).$ Using the orthogonality of $\{v_1, v_2, v_3, v_4 \},$ we find that  $\phi(v_3, v_3)=0$ and $\phi(v_4, v_4)=0.$ We now claim that $ \langle  v_3 \rangle  = \langle  v_4\rangle .$ If $ \langle v_3\rangle  \neq \langle v_4\rangle, $ then $W$ has $\{v_3, v_4\}$ as its basis. Choose $v \in V \setminus  W.$ Clearly, $\{ v_3, v_4, v\}$ is a basis of $V.$ Let $ \phi(v_3, v) = a$ and $\phi(v_4, v) = b.$ Consider the vector $ w= -b v_3 + a v_4$ in $W.$ As noted above,  $\phi(v_3, v_3)= \phi(v_4, v_4)=0$ and hence  we obtain  $\phi(w, v_3) = \phi(w, v_4)=0.$ Observe that
		$$  \phi(w, v) = -b\phi(v_3, v) + a \phi(v_4, v) =-ba+ab =0.$$ Since $\{ v_3, v_4, v\}$ is a basis of $V,$ we get $w \in V^{\perp}=0.$ So $ w=-b v_3 + a v_4=0.$ Since $ \{ v_3, v_4\}$ is linearly independent we have  $a=b=0.$ Now, as  $ \phi(v_3, v) = a$ and $\phi(v_4, v) = b,$ we get that $\{ v_3, v_4, v\}$ is an orthogonal basis of $V.$ Note here that $v_3$ and $v_4$ are self orthogonal vectors. If we let $A$ denote the matrix representation of $\phi$ with respect to this orthogonal basis, then we see that $A$ is a diagonal $3 \times 3 $ matrix
		with diagonal entries $\phi(v_3, v_3)=0, \phi(v_4, v_4)=0$ and $ \phi(v,v).$  Hence $A$ is singular, contradicting the fact that $\phi$ is non-degenerate. Hence  $ \langle  v_3 \rangle  = \langle  v_4 \rangle $ and we are done in this case too. 
		
	\end{proof}
	
	\begin{question}\label{question}
		Let $V$ be a vector space over field $k$ with $\dim V \geq  4$ and let $\phi$ be a non-degenerate bilinear form on $V.$  Let $ S= \{v_1, v_2, \ldots, v_m \}$ be an orthogonal set of vectors in $V$ with respect to $\phi$ such that $ m > \dim V.$\\
		(1)	Then is it true that $ v_i = \alpha v_j$ for some $ i \neq j?$ \\
		(2) More generally, is it true that there exists a subset $T \subset S$ such that $|T|=n$ and for any $ v_i \in S \setminus T $ there is a vector $v_j \in T$ such that $ v_i = \alpha v_j$  for some $\alpha \in k?$
	\end{question}
	
	\section{Main Results}
	
	By a ring $R,$ we mean a commutative ring with unity. If in addition, $R$ is local ring with unique maximal ideal $\m$ and the residue field $k=R/\m,$ then we denote $R$ by $(R, \m, k).$ We first recall some basic definitions and preliminary facts used  throughout the paper. 
	\begin{definition} Let $(R, \m, k)$ be a Noetherian local ring. The ideal $(0:_R \m)$ is called the \textit{socle} of $R.$ This is denoted by $\soc R.$ Note that $\soc R$ is a finite dimensional vector space over the residue field $k =R/\m.$  
		
	\end{definition}

	\begin{definition} Artinian local ring $(R, \m, k)$
		is called an \textit{Artinian Gorenstein} local ring, if $\dim_{k} ( 0 :_R
		\m)=1,$ i.e., $\soc R$ is $1$-dimensional vector space over field $k.$
	\end{definition}

	\begin{definition} \textit{ Hilbert function} of a Noetherian local ring $(R, \m, k)$ is the numerical function $H_R: \mathbb N_{ \geq 0}  \longrightarrow \mathbb N_{ \geq 0}$ given by, for $ i \geq 0,$ $$  H_R(i) = \dim_k (\m^i/\m^{i+1})$$  The dimension $H_R(1)$ of the $k$-vector space $\m/\m^2$ is called the \textit{embedding dimension} of $R,$ denoted by $ \embdim R.$\
	\end{definition}
	
	\begin{remark}\label{length} Let $(R, \m, k)$ be  an Artinian local ring. \\
		(i) Let $n$ be the smallest positive integer with $\m^n \neq 0.$ Then $H_R(i)=0$ for $i \geq n+1.$ \\
		(ii)  Since $\ell(\m^i/\m^{i+1})= \dim_k (\m^i/\m^{i+1}),$  we notice that $\ell(R) = \displaystyle \sum_{i=0}^{n}H_R(i).$
		
	\end{remark}

	\begin{observation}\label{obs}
		Let $R$ be a ring with $\ell(R) =5.$ In view of the Remark \ref{length}(ii), it is easy to see that we have only the following possible Hilbert functions:
		
		\begin{enumerate}
			\item [\rm (i)] $ H_R(i)= 1$ for $ 0 \leq i \leq 4;$ 	
			\item [\rm (ii)]  $ H_R(0)= 1, ~ H_R(1)=3, ~ H_R(2)=1;$
			\item [\rm (iii)] $ H_R(0)= 1,~  H_R(1)=2, ~ H_R(2)=2;$
			\item [\rm (iv)] $ H_R(0)= 1, ~ H_R(1)=2, ~ H_R(2)=1, ~ H_R(3)=1.$
		\end{enumerate}
		We also see that if $(R, \m, k)$ has a Hilbert function of the case (i), $\m$ is a principal ideal. So $R$ is a principal ideal ring (PIR) in this case.
	\end{observation}

	\noindent  We first prove a key result in our analysis of zero-divisors of length five rings and will be used frequently throughout the paper.
	
	\begin{proposition}\label{bilinear}
		Let $(R, \m, k)$ be an Artinian local ring with $\m^2=(l) \neq 0 $ and $\m^3=0.$ Then
		\begin{enumerate}
			\item [\rm (i)] there exists a symmetric bilinear form \[\phi : \m/\m^2 \times \m/\m^2 \rightarrow k \] such that  $\phi (\bar a, \bar b)=0 $ if and only if $ab=0$ in $R.$	
			\item [\rm (ii)] $ a\in \soc R$ if and only if $\bar a \in (\m/\m^2)^{\perp}$ and $\dim_k \soc R = \dim_k (\m/\m^2)^{\perp}+1.$
			\item [\rm (iii)] $R$ is Gorenstein if and only if $\phi$ is non-degenerate.
		\end{enumerate}

	\end{proposition}
	\begin{proof}
		First note that, since $ \m^2 =(l)$ and $\m^3=0$ we get, $ l \in (0: \m) =\soc R.$\\	
		(i) Now define the map $\phi$ as follows. For $\bar a$ and $\bar b$ in $\m/\m^2,$  $ab \in \m^2=(l),$ hence $ab = ul$ for some $ u \in R.$ Define $\phi(\bar a, \bar b) = \bar u \in k=R/\m.$ We now check that $\phi$ is a well defined map. It is enough to check the well definedness of $\phi$ in any one of the two variables. To do  this, we let $\bar a = \bar a_1$ in $\m/\m^2$ and suppose $ab=ul$ and $a_1b = vl$ for some $u, v \in R.$ So $ a-a_1 \in \m^2= (l),$ giving us  $ a-a_1 = wl$ for some $w \in R.$ Multiplying both sides by $b,$ we get $ab-a_1b = wlb.$ Since $ l \in (0:\m), $ this gives $lb=0$ and hence $ab-a_1b=0,$ i.e., $ ab= a_1b$ in $R.$ This gives us  $ul=vl $ in $\m^2.$ Now, as $\m^2$ is a one dimensional $k$-vector space, we have the equality $ \bar u \bar l = \bar v \bar l$ in $\m^2$ treated as a $k$-vector space. Thus $\bar u = \bar v$ in $R/\m =k.$ Hence $\phi ( \bar a , \bar b) = \phi (\bar a_1, \bar b).$  While bilinearity of $\phi$ is clear, symmetry follows from the  commutativity of $R.$ Thus $\phi$ is a symmetric bilinear form. Clearly, if $ab=0$ then $\phi(\bar a, \bar b) =0.$ Conversely suppose that $\phi(\bar a, \bar b) = 0.$ This means that $ab = ul$ with $ \bar u = 0$ in $R/\m=k,$ i.e., $ u \in \m.$ If $u=0$ then $ab=0.$ If $ u \neq 0$ in $\m$ then $ul=0$ as $l \in (0: \m), $ and hence $ab=0$ again. Thus $ab=0$ is equivalent to $\phi(\bar a, \bar b) = 0,$ i.e., to $\bar a$ being orthogonal to $\bar b$ with respect to symmetric bilinear form $\phi.$\\
		\noindent (ii) Let $a \in \soc R,$ i.e., $ab=0$ for all $ b \in \m.$ By (i), this is equivalent to  $\phi(\bar a, \bar b) = 0$ for all $\bar b \in \m/\m^2,$ i.e., $ \bar a \in (\m/\m^2)^{\perp}. $ Thus $a \in \soc R$ if and only if $\bar a \in (\m/\m^2)^{\perp}.$ We  define a map \[f : \soc R \longrightarrow (\m/\m^2)^{\perp}\] given by $f(a) = \bar a \in \m/\m^2.$  It is easy to see that $f$ is a $k$-linear map. As noted above if $a  \in (\m/\m^2)^{\perp}$ then $ a \in \soc R.$ This proves that $f$ is surjective.  Since $\m^2=(l),$ we find that $\ker f = \langle l \rangle ,$ a one dimensional subspace. Thus we have a short exact sequence of $k$-vector spaces,
		\[ 0 \longrightarrow \langle l \rangle  \longrightarrow \soc R \longrightarrow (\m/\m^2)^{\perp} \longrightarrow 0 .\]  This proves  that $\soc R \cong (\m/\m^2)^{\perp}  \oplus \langle l \rangle  $ and that $\dim_k \soc R = \dim_k (\m/\m^2)^{\perp}+1.$\vspace{0.1 cm} \newline	
		\noindent (iii) Recall that  $R$ is Gorenstein if $\dim_k \soc R =1$ and  in our case this is equivalent to  $\soc R =(l) = \m^2. $ We also note that bilinear form $\phi$ being  non-degenerate is equivalent to $(\m/\m^2)^{\perp} =0.$ The result is now immediate from (ii). 
	\end{proof}
	\begin{lemma}\label{socle}
		Let $(R, \m , k)$ be an Artinian local ring. 
		\begin{enumerate}
			\item [\rm (i)] If  $\ell(R) =5 $ and $H_R(1) =3,$  then $ 1 \leq  \dim_k \soc R \leq 3.$ 
			\item[\rm (ii)] Suppose  $\m^2 \neq 0$ and $\m^3=0.$ For $a , b \in \m,$ if $ \ann(a) \neq \ann(b), $ then $\langle \bar a \rangle \neq  \langle \bar b \rangle $ as one dimensional subspaces in $\m/\m^2.$ 
			
		\end{enumerate}
		
	\end{lemma}
	\begin{proof} (i) Since $\ell(R)=5$ and  $H_R(1)=3,$ the Hilbert function of $R$ is $H_R(0)=1, H_R(1) =3, H_R(2) =1$ and $H_R(i )=0 $ for $ i \geq 3.$ Hence we have $\m^2= (l) \neq 0$ and $\m^3=0.$ As $\m^2 = (l) \subseteq \soc R,$ we get $ \dim \soc R \geq 1.$ Consider the following short exact sequence of $R$-modules,
		\[ 0 \rightarrow \m/\soc R \rightarrow R/ \soc R \rightarrow R/\m \rightarrow 0\]
		Hence from the additivity of length, we obtain, $ \ell(R/\soc R) = \ell(R/\m) + \ell(\m/ \soc R).$ Since $\ell(R) < \infty,$ we get $\ell(R/ \soc R) = \ell(R) - \ell(\soc R).$ Hence, \[ \ell(R) = \ell(R/\m) + \ell(\m/ \soc R) + \ell(\soc R)   \] We now observe that the inclusion $\soc R \subsetneqq \m$ is  strict.   This is because, if $\soc R = \m,$  then $\m^2 =0,$ a contradiction. So $\ell(\m/\soc R) \geq 1.$ Thus length of each term in the above equation  is at least one.  Note also  $\ell(\soc R) = \dim_k \soc R.$ Hence if $\dim_k \soc R \geq 4, $ then $\ell(R) \geq 6,$ a contradiction. Therefore $\dim_k \soc R \leq 3.$ \vspace{0.1 cm} \newline
		(ii) Note first that $\m^2 \subseteq 0:\m.$ Suppose on the contrary that $\langle \bar a \rangle  = \langle  \bar b\rangle .$ Then $a-ub = c \in \m^2$ for some unit $u \in R.$ We now show that $\ann(a) = \ann(b).$ To do this, let $ x \in \ann(a). $ Multiplying the equation $a-ub = c$ by $x$ on both sides, we get $ xa - xub = xc,$ i.e., $-xub=xc.$ Since $ x\in \m$ and $ c \in \m^2 \subseteq 0: \m$ we obtain, $xc=0.$ Thus $ xub =0,$ and as  $u$ is unit in $R,$ we see that  $x \in \ann(b).$ This shows $\ann(a) \subseteq \ann(b).$ Other inclusion is similar. Thus $\ann(a) = \ann(b).$
	\end{proof}
	
	As the first application of Lemma \ref{socle}(ii), we give an easy proof of the fact that $ \omega(\Gamma_E(R)) < \infty $ for rings of length $4.$ The original proof by J. Coykendall et al. \cite[Proposition 5.8(v)]{Coy} is a case by case analysis and proceeds by obtaining presentations of rings using  Cohen structure theorem in some specific cases. The most complicated case is when $R$ is a local ring of length $4.$ We offer an easy proof of their result, avoiding, for instance, Cohen's structure theory. Other cases where $\ell(R) \leq 3$ can easily be dealt with by first proving the result in the local case and then in the general case writing $R$ as a finite product of Artinian local rings. 
	
	\begin{proposition}
		Let $(R, \m, k)$ be a local ring of length $4.$ Then $\omega(\Gamma_E(R)) \leq 3.$ 	
	\end{proposition}
	\begin{proof}
		
		Note that we have three possibilities for the Hilbert function of local ring $R$ of length $4,$ namely:
		\begin{enumerate}
			\item [\rm (i)] $ H_R(i)= 1$ for $ 0 \leq i \leq 3$ 	 and  $H_R(i)=0$ for  $i \geq 4;$
			\item [\rm (ii)]  $H_R(0)= 1, ~ H_R(1)=3, $ and $ H_R(i)=0$ for all $ i \geq 2;$
			\item [\rm (iii)] $ H_R(0)= 1,~  H_R(1)=2, ~ H_R(2)=1$ and $H_R(i)=0$ for all $ i \geq 3.$
		\end{enumerate}
		(i) In the case (i), $R$ is a principal ideal ring because $\m$ is a principal ideal. Hence, in this case, $\Gamma_E(R)$ is a finite graph consisting of exactly three vertices corresponding to generators of powers of  ideal $\m, \m^2$ and $ \m^3.$ Hence  $\omega(\Gamma_E(R)) \leq 3.$ \vspace{0.1 cm} \newline
		(ii) In this case, we have $\m^2=0$ and so $\ann(x) =\m$ for all $x \in \m.$ Hence $\Gamma_E(R)$ consists of exactly one vertex and the conclusion is obvious. \vspace{0.1 cm} \newline
		(iii) In this case, $\dim_k \m/\m^2=2,  \m^2 = (l)$ and $\m^3=0.$ Suppose on the contrary that $\omega(\Gamma_E(R)) \geq 4.$ Observe that a clique of size $4$ in $\Gamma_E(R)$ must contain subclique of size $3,$ say, $\{[a_1], [a_2], [a_3]\}$ such that $[a_i] \neq [l]$ for any $i.$ By lemma \ref{socle}(ii), $\langle \bar a_i \rangle  \neq \langle \bar a_j \rangle $ for $ i \neq j.$ Since $\dim_k \m/\m^2 =2,$  $\m/\m^2 = \langle  \bar a_1, \bar a_2 \rangle $ and hence by Nakayama lemma $ \m=(a_1, a_2).$ Since $ a_3a_1 = a_3a_2 =0,$ we get $ a_3 \m =0,$ i.e., $a_3 \in \soc R.$ This gives $[a_3]=[l]$  contradicting $[a_i] \neq [l] $ for any $i.$ Thus $\omega(\Gamma_E(R)) \leq 3$ in this case too.

	\end{proof}
	
	

	\begin{proposition}\label{bound}
		Let $(R, \m , k)$ be an Artinian local ring with $\ell(R) =5 $ and $H_R(1)=3.$ Then $ \omega(\Gamma_E(R)) < \infty.$ In fact, $\omega (\Gamma_E(R))  \leq  5- \dim_k \soc R.$
	\end{proposition}
	\begin{proof} 
		Since $\ell(R)=5$ and  $H_R(1)=3,$ we have $\m^2= (l) \neq 0$ and $\m^3=0.$ By  Proposition \ref{bilinear}, there exists a symmetric  bilinear form \[\phi: \m/\m^2 \times \m/\m^2 \rightarrow k \] such that for $a, b \in \m,$ we have  $\phi (\bar a, \bar b)=0 $ if and only if $ab=0$ in $R.$ By Lemma \ref{socle}(i), we have  $1 \leq \dim_k \soc R \leq 3.$ We now divide the proof in three cases according as $\dim \soc R$ is $1, 2$ or $3.$ \vspace{0.15 cm} \newline
		\noindent {\bf Case(i):} $\dim_k \soc R =1.$ In this case, $R$ is Gorenstein and by Proposition \ref{bilinear}(iii), the symmetric bilinear form $\phi$ is non-degenerate.  Since $ l \in \soc R,$ we obtain that vertex $[l]$ is adjacent to every other vertex in $\Gamma_E(R).$ Also since $\soc R = \m^2= (l) $ and $\m^3=0,$ hence for any $[x] \in V(\Gamma_E(R)) \setminus \{ ~ [l] ~ \},$ we have  $ x \in \m \setminus \m^2.$ \\ 
		\indent We show that  $ \omega (\Gamma_E(R)) \leq 4.$ Suppose on the contrary that $ \omega (\Gamma_E(R)) \geq 5.$  In any clique of size $5,$ discarding the vertex $[l]$ if present in the clique, we are left with a clique of size $4,$ say $\{ [a_1], [a_2], [a_3], [a_4]\}$ such that $ a_i \in \m \setminus \m^2.$ Since $a_i a_j =0$ for $ i \neq j,$ by Proposition \ref{bilinear}(i) we obtain $ \phi(\bar a_i, \bar a_j) =0$ for all $ i \neq j.$ Thus $ \{ \bar a_1, \bar a_2, \bar a_3, \bar a_4 \}$ is an orthogonal set in $\m /\m^2$ w.r.t. bilinear form $\phi.$ Since $\phi$ is a non-degenerate bilinear form on $\m/\m^2$  and $\dim \m/\m^2=3,$ by Lemma \ref{ortho}, we obtain, $ \langle \bar a_i \rangle  = \langle \bar a_j \rangle $ for some $ i \neq j.$ Hence by Lemma \ref{socle}(ii),  $ \ann(a_i) = \ann(a_j).$ This gives $[a_i] =[a_j]$ contradicting the distinctness of the vertices in the clique $\{ [a_1], [a_2], [a_3], [a_4]\}.$  Hence $ \omega(\Gamma_E(R)) \leq 4.$ \vspace{0.15 cm} \newline
		\noindent {\bf Case(ii):} $\dim_k \soc R=2.$ By Proposition \ref{bilinear}(ii), we have $\dim_k (\m/\m^2)^{\perp}=1.$ In this case $\phi$ is a degenerate bilinear form  with $\rank \phi =2.$ Since $\dim_k (\m/\m^2)^{\perp}=1, $  $(\m/\m^2)^{\perp} =  \langle  \bar a \rangle $ for some $ a \in \m \setminus \m^2.$ By Proposition \ref{bilinear}(ii), $ a \in \soc R$ and so $[a]=[l]$ in $\Gamma_E(R).$ In fact, for any $0 \neq b \in \soc R$  we see that $ [b]= [l]$ in $\Gamma_E(R).$ \vspace{0.1 cm} \newline  We now claim that $\omega(\Gamma_E(R)) \leq 3.$ Suppose on the contrary that $\omega(\Gamma_E(R)) \geq 4.$ Observe that a clique of size $4$ must have a subclique $\{[x], [y], [z] \}$ such that $ x, y, z \notin \soc R.$ Since $ \m^2 \subseteq \soc R,$  $ x, y, z \notin \m^2.$
		We make the following observations. \vspace{0.1 cm} \newline
		\noindent (i) $ x, y, z \notin \m^2 \implies  \bar x, \bar y, \bar z$ are non-zero vectors in $\m/\m^2.$\\
		(ii)  By Proposition \ref{bilinear}(ii), $  \{\bar x, \bar y, \bar z \}$ is an orthogonal set  in $\m/\m^2.$ \\
		(iii) $ x, y, z \notin \soc R \implies \bar x, \bar y, \bar z \notin (\m/\m^2)^{\perp}$ ( by Proposition \ref{bilinear}(ii)).\\
		(iv)  By Lemma \ref{socle}(ii), $ \langle  \bar x \rangle ,  \langle \bar y \rangle , \langle \bar z\rangle  $ are mutually distinct  subspaces of $\m/\m^2.$  \\
		(v) By (iv), it follows that   $\dim_k \langle  \bar x, \bar y, \bar z \rangle $ is either $2$ or $3.$ \vspace{0.1 cm} \newline
		\noindent Suppose first that $ \dim_k \langle  \bar x, \bar y, \bar z \rangle =3.$ Since $\dim_k \m/\m^2=3,$ we get $\m/\m^2 = \langle  \bar x, \bar y, \bar z \rangle.$ Hence by Nakayama lemma, $\m= (x, y, z).$ By observation (ii) above, $\{\bar x, \bar y, \bar z \}$ is an orthogonal basis of $\m/\m^2.$ Since $\phi$ is degenerate, hence the matrix of $\phi$ corresponding to orthogonal basis $\{\bar x, \bar y, \bar z \}$ is a diagonal singular matrix. Hence one of $\bar x, \bar y, \bar z,$ say $\bar x,$ has the property that $ \phi(\bar x, \bar x)=0.$ As $\phi(\bar x, \bar y) = 0=\phi(\bar x, \bar z),$  by Proposition \ref{bilinear}(i), $ xx=xy=xz=0.$ Now since $\m = (x, y, z),$ we have $x\m=0,$ i.e., $ x\in (0:\m).$ This is a  contradiction to our choice of $x \notin \soc R.$ Hence  $ \dim_k \langle  \bar x, \bar y, \bar z \rangle =3$ is not possible. \\
		\indent Now, we can assume that  $ \dim_k  \langle  \bar x, \bar y, \bar z \rangle =2.$ By the observation (iv) above, none of $\bar x, \bar y$ or $\bar z$ is a multiple of the other, so we see that any two of $\bar x, \bar y, \bar z$ forms a basis of $\langle  \bar x, \bar y, \bar z \rangle ,$ say $\{ \bar x, \bar y\}.$ Hence $\bar z = \bar u \bar x + \bar v \bar y$ for some $ v, v \in R$ such that $\bar u \neq 0, \bar v \neq 0$ in $k =R/\m.$ Multiplying both sides by $\bar x $  we find that  $\bar x^2 =0.$ Similarly $\bar y^2 =0.$ Thus $\phi(\bar x, \bar x) = \phi(\bar y, \bar y)=0.$\vspace{0.1 cm} \newline
		We now claim that,  \[ \m/\m^2 \cong \langle  \bar a \rangle  \oplus \langle  \bar x, \bar y \rangle. \]
		Since $\dim \m/\m^2=3,$ to prove our claim, it suffices to show that $\langle  \bar a \rangle  \cap \langle  \bar x, \bar y \rangle  = 0.$ Suppose $ \langle  \bar a \rangle  \cap \langle  \bar x, \bar y \rangle  \neq  0,$ i.e., $\bar a \in  \langle  \bar x, \bar y\rangle .$   Then $\bar a  = \bar \alpha \bar x + \bar \beta \bar y$ for some $ \alpha , \beta \in R.$ We first observe that both $\bar \alpha \neq 0 $ and $\bar \beta \neq 0$ in $k.$ This is because, if  $\bar \alpha =0,$ then $ \bar a = \bar \beta \bar y.$ This gives $ \langle \bar a \rangle  = \langle  \bar y \rangle ,$ and hence by Lemma \ref{socle}(ii), $ \ann(a) = \ann(y).$ As $a \in \soc R$ we have $\ann(a) =\m$ and  so $\ann(y)=\m.$ This gives $ y \in \soc R,$ a  contradiction to  $y \notin \soc R.$ Thus both   $\bar \alpha \neq 0, \bar \beta \neq 0$ in $k$ and therefore  $\alpha, \beta $ are units in $R.$ Now  $\bar a  = \bar \alpha \bar x + \bar \beta \bar y$ gives us  $ a - \alpha x - \beta y \in \m^2=(l). $ We write,  
		\[\hspace{2 cm} a - \alpha x - \beta y  = u l  \hspace{2 cm}  \quad . ~ . ~ . ~ . ~ . ~ . ~ (A) \] for some $ u \in R.$ We now show that $\ann(x) = \ann(y).$ To do this, let $ t \in \ann(x). $ Multiplying the equation (A) above by $t$ on both sides we get $ ta -t \alpha x - t\beta y  = tu l.$
		As $ a, l \in \soc R,$  $ta= tl=0.$  The above equation  becomes $t \beta y =0.$ Since $\beta$ is a unit, we obtain $ty=0$ giving us $t \in \ann(y).$ Hence $ \ann(x) \subseteq \ann(y).$ Similarly, we can show that $\ann(y) \subseteq \ann(x).$ Thus $\ann(x) = \ann(y),$ i.e., $[x]=[y]$ contradicting the distinctness of the vertices  $[x]$ and $[y]$ of $\Gamma_E(R).$ Hence  $ \langle  \bar a \rangle  \cap \langle  \bar x, \bar y \rangle  = 0.$ This proves our claim  $\m/\m^2 \cong \langle  \bar a \rangle  \oplus \langle  \bar x, \bar y \rangle  \cong (\m/\m^2)^{\perp} \oplus \langle  \bar x, \bar y \rangle .$ Thus we have, \[ \frac{\m/\m^2}{ (\m/\m^2)^{\perp} } \cong \langle  \bar x, \bar y \rangle \]
		It follows from Theorem \ref{symbi}(ii) that  $\langle  \bar x, \bar y \rangle $ is a non-degenerate subspace of $\m/\m^2$ corresponding to $\phi.$ Hence $\phi(\bar x, \bar x)\neq 0$ and $\phi(\bar y, \bar y)\neq 0.$ This contradicts $\phi(\bar x, \bar x) = \phi(\bar y, \bar y)=0.$ Hence we have $\omega(\Gamma_E(R)) \leq 3.$\vspace{0.15 cm} \newline
		\noindent {\bf Case(iii):} $\dim_k \soc R =3.$ We claim that $ \omega (\Gamma_E(R)) =2.$  We suppose on the contrary that $\omega(\Gamma_E(R)) \geq 3.$ As before, for every element $0 \neq t \in\soc R,$ we have $[t]=[l]$ in $\Gamma_E(R).$ So a clique of size $3$ must have a subclique of size $2,$ i.e., adjacent vertices $[x]$ and $[y]$ in $\Gamma_E(R)$ such that  $ x, y \notin \soc R.$ Since $[x]$ adjacent to $[y]$ in $\Gamma_E(R),$ we have $xy=0.$ Since $\dim_k \soc R=3,$ by Proposition \ref{bilinear}(ii),  $\dim (\m/\m^2)^{\perp} =2.$ Also as $ y \notin \soc R$  Proposition \ref{bilinear}(ii) gives us $\bar y \notin (\m/\m^2)^{\perp}.$ If we let $\{\bar a_1, \bar a_2\}$ be a basis for $ (\m/\m^2)^{\perp},$ then $\{\bar a_1, \bar a_2, \bar y \}$ is a basis for $\m/\m^2.$ By Nakayama lemma $ \m = (a_1, a_2, y).$
		Since $\bar a_1, \bar a_2 \in (\m/\m^2)^{\perp},$ hence by Proposition \ref{bilinear}(ii), we have, $ a_1, a_2 \in \soc R.$  Thus $ x a_1= xa_2= xy =0$ implying $ x\in 0:\m= \soc R.$ This is a contradiction to $x \notin \soc R.$ Hence  $\omega(\Gamma_E(R)) \leq 2.$ Further, note that $[l]$ is adjacent to every vertex in $\Gamma_E(R).$ Since there exists $ t \notin \soc R,$ we find that $[t] \neq [l],$ and that there is an edge joining $[t]$ and $[l].$  Hence, $ \omega (\Gamma_E(R)) =2$ in this case.
		
	\end{proof}

	\begin{corollary} Let $(R, \m, k) $ be an Artinian local ring with $\ell(R)=5$ and $H_R(1) =3.$ If $\cha k \neq 2,$  then $ \omega (\Gamma_E(R))  =  5- \dim_k \soc R.$
	\end{corollary}
	\begin{proof}
		By Proposition \ref{bound}, we have $\omega (\Gamma_E(R))  \leq   5- \dim_k \soc R.$
		Let $W$ be a complement of $(\m/\m^2)^\perp$ in $\m/\m^2,$ i.e., $\m/\m^2 \cong (\m/\m^2)^\perp \oplus W.$  Let $\dim W=s.$ Note that the bilinear form $\phi$ (given by Proposition \ref{bilinear}), restricted to $W$ is non-degenerate. Since $\cha k \neq 2,$  by Proposition \ref{symbi}(ii), there is an  orthogonal basis $\{ a_1, a_2, \ldots a_s\}$ for the subspace  $W.$ By Proposition \ref{bilinear}(i), $a_ia_j=0$ and since $\phi$ is non-degenerate on $W,$ we have $a_i^2 \neq 0$ for all $i.$ We next observe that the vertices $[a_i]$ are distinct, i.e., $ [a_i] \neq [a_j]$ for $ i \neq j.$ This is because if $[a_i]= [a_j],$  then $\ann(a_i )=\ann(a_j).$ Since $ a_i \in \ann(a_j) =\ann(a_i),$ we get $a_i^2=0$ contradicting $a_i^2 \neq 0.$ Hence $ [a_i] \neq [a_j]$ for $ i \neq j.$ As $a_ia_j=0,$ we have $ [a_i]$ is adjacent to $[a_j]$ in $\Gamma_E(R)$ for $ i \neq j.$ Hence $\{[a_1], \ldots, [a_s]\}$ forms a clique of size $ \dim W= s.$ Since the vertex $[l]$ is adjacent to all vertices in $\Gamma_E(R),$ we get a clique of size $s+1,$ namely, $\{[a_1], \ldots, [a_s], [l]\}.$ Now,
		\begin{eqnarray*}
			s+1 & = & \dim W+1 \\ &=& 3- (\dim \m/\m^2)^\perp +1 ~~ 
			( \text{since} ~  \m/\m^2 \cong (\m/\m^2)^\perp \oplus W)\\
			&=& 4 -(\dim \m/\m^2)^\perp\\
			&=& 5- \dim_k \soc R  ~~ \quad  \text{(by Proposition \ref{bilinear}(ii)).}
		\end{eqnarray*}
		Thus, we have a clique of size $5- \dim_k \soc R$ in $\Gamma_E(R),$ 
		proving the result.
	\end{proof}
	
	\begin{proposition}\label{bound1}
		Let $(R, \m, k)$ be an Artinian local ring of length $5$ and  Hilbert function $H_R(0)=1, H_R(1) =2, H_R(2)=2$ and $H_R(i) = 0 $ for $ i \geq 3. $ Then $\omega(\Gamma_E(R)) \leq 3.$  In fact, $\omega(\Gamma_E(R)) \leq 5- \dim_k \soc R.$
	\end{proposition}
	\begin{proof}
		Since $\m^2 \subseteq \soc R,$ there are only two cases to consider, namely, $\dim_k \soc R =2$ and $\dim_k \soc R=3.$  \vspace{0.15 cm} \newline 
		Case(i): $\dim_k \soc R =2.$ In this case, we show that $\omega(\Gamma_E(R)) \leq 3.$
		Suppose on the contrary that $\omega(\Gamma_E(R)) \geq 4.$ 
		As before  for any $s, t \in\soc R,$ both non-zero, we have $\ann(s) = \ann(t)=\m,$ i.e., $[s]=[t]$ in $\Gamma_E(R).$ Now discarding the vertex coming from $\soc R,$ if occurring in the clique of size $4,$ we are left with a subclique of three vertices $\{[a], [b], [c] \}$ such that $ a, b, c \notin \soc R.$ Since $\m^3 =0,$ by Lemma \ref{socle}(ii), the one dimensional subspaces $ \langle \bar a \rangle ,  \langle \bar b \rangle , \langle  \bar c \rangle $ are distinct. Now as $\dim_k \m/\m^2=2,$ we get $\m/\m^2 = \langle  \bar a, \bar b, \bar c \rangle .$ In fact, any two of $\bar a, \bar b, \bar c$ will form a basis of $\m/\m^2,$ say $\{\bar b, \bar c\}.$ By Nakayama lemma, we get $\m = (b, c).$ Since $[a]$ is adjacent to $[b]$ and $[c],$ we have $ab=ac=0$ giving us $ a \in 0 : \m =\soc R.$ This is a contradiction to the fact that $a \notin \soc R.$ Hence $\omega(\Gamma_E(R)) \leq 3.$\vspace{0.15 cm} \newline 
		Case (ii):  $\dim_k \soc R =3.$ In this case, we show that $\omega(\Gamma_E(R)) \leq 2.$ Note that in this case $ \m^2 \subsetneqq \soc R.$ So choose $ s\in \soc R \setminus \m^2.$ Now suppose on the contrary that  $\omega(\Gamma_E(R)) \geq 3.$ Again as in Case(i), discarding the vertex coming from $\soc R$ if occurring in the given clique,  we are left with two adjacent vertices $[a]$ and $[b]$ such that $a, b \notin \soc R.$ Adjacency of $[a]$ and $[b]$ gives $ab=0.$ Since $[a] \neq [s],$ by Lemma \ref{socle}(ii), it follows that the one-dimensional subspaces  $ \langle \bar a \rangle \neq   \langle \bar s \rangle $ in $\m/\m^2.$ Since $\dim_k \m/\m^2=2$ we get $\m/\m^2 = \langle \bar a, \bar s \rangle.$ By Nakayama lemma, $\m = (a, s).$ Since $ ba=bs=0,$ $b \in 0: \m = \soc R.$ This is a contradiction to $ b \notin \soc R.$ Hence $\omega(\Gamma_E(R)) \leq 2.$ Thus in both the cases, we have  $\omega(\Gamma_E(R)) \leq 5- \dim_k \soc R.$
		
	\end{proof}
	\begin{proposition}\label{bound2}
		Let $(R, \m, k)$ be an Artinian local ring of length $5$ and  Hilbert function $H_R(0)=1, H_R(1) =2, H_R(2)=1, H_R(3)=1$ and $H_R(i) = 0 $ for $ i \geq 4. $
		Then $\omega(\Gamma_E(R)) < \infty.$ In fact,  $\omega(\Gamma_E(R)) \leq 5- \dim_k \soc R.$
	\end{proposition}
	\begin{proof}
		Since $H_R(2)=H_R(3)=1,$ we have  $\m^2 =(l)$ and $\m^3= (k)$ with $ l, k \neq 0$ in $R.$ We  first observe that $ l \notin \soc R.$ To do this, suppose $ l \in \soc R.$ Since $\m^3 = (k) \subseteq \m^2=(l),$ we obtain $ k = rl$ for some $r \in R.$ If $ r \in \m,$ then as $ l \in \soc R,$ we have $k=rl =0,$ a contradiction to $ k \neq 0.$ Hence $ r \notin \m,$ i.e., $r$ is a unit. But then, $l= r^{-1}k$ giving us $ \m^2=\m^3.$ By Nakayama lemma this $\m^2 =0,$  contradicting $l \neq 0.$ This shows that $l \notin \soc R$ and so $[l] \neq [k].$ To complete the proof we make two cases, namely, $\dim_k \soc R =1$ and $\dim_k \soc R \geq 2.$  \vspace{0.15 cm} \newline 
		{\bf Case(i):} $\dim_k \soc R =1.$  In this case, we have  $ \soc R = \m^3 = (k).$  Notice that elements of $\m^2$ and $\m^3$ contribute to only two vertices namely, $[l]$ and $[k]$ in $\Gamma_E(R).$ We first claim that,  \vspace{0.1 cm} \newline
		\noindent \underline{ Claim 1}: If $ \{ [a], [b], [c]\}$ is a clique with  $ a, b, c \in \m \setminus \m^2,$ then   $\dim_k \langle  \bar a, \bar b, \bar c \rangle  =2.$ \vspace{0.15 cm} \newline 
		Note that since $\soc R =(k) \subseteq \m^2 $ and $ a, b, c \notin \m^2,$ we have,  $ a, b, c \notin \soc R.$ As $H_R(1)=2$ (i.e. $\dim_k \m/\m^2 =2 $),  $\dim_k \langle  \bar a, \bar b, \bar c \rangle  \leq 2.$ Suppose on the contrary that  $\dim_k \langle  \bar a, \bar b, \bar c \rangle  =1.$ So we obtain $\bar a= \bar u_1 \bar b$ in $\m/\m^2$ and $\bar a = \bar u_1 \bar c$ for $ \bar u_1, \bar u_2$ non-zero in $k.$ Hence $a-u_1b = r_1 l$ and $a-u_2c = r_2 l$ for some $r_1, r_2 \in R$ and units $u_1, u_2 \in R.$   \newline  We now show that $r_1$ and $r_2$ are units in $R.$ Suppose on the contrary that $r_1 \in \m.$ Then  $r_1l \in \m^2=(k).$ Thus $r_1l = ck$ for some $c \in R$ and  we have $a-u_1b = ck.$ We now observe that $\ann(a) = \ann(b)$ as follows. Let $ x \in \ann(a).$ Hence $ xa - xu_1b = xk.$ Since $ k \in \soc R,$ $x k =0$ and $ x \in \ann(a)$ gives $ax=0.$  Hence $ xu_1 b=0.$ As $ u_1$ is a unit in $R,$  $xb =0,$ showing that $ x \in \ann(b).$ Hence $ \ann(a) \subseteq \ann(b).$ Argument for $ \ann(b) \subseteq \ann(a)$ is similar. This proves $ \ann(a) = \ann(b),$ i.e., $[a] =[b].$ This is a contradiction to the distinctness of $[a]$ and $[b].$ Thus $r_1$ is a unit in $R.$ Similarly, we can show that $r_2$ is a unit. \newline 
		\indent Now, multiplying the two equations $a-u_1b = r_1 l$ and $a-u_2c = r_2 l$ together and noting that $ l^2 \in  \m^4=(0)$ we obtain, \[(a-u_1b)  (a-u_2c) = r_1r_2 l^2=0.\]
		Since $ \{ [a], [b], [c]\}$ is a clique, we have $ab=bc = ac =0.$ This gives $a^2=0.$ We can similarly have $ b^2=c^2=0.$ Multiplying  $a-u_1b = r_1 l$ by $a$ on both sides gives us $ar_1l=0.$ As $r_1$ is a unit, we get $al=0.$ Similarly,  $ bl=cl=0.$ Now  $[a] \neq [b] $ gives $\ann(a) \neq \ann(b).$ Hence there exists $ x \in \ann(a) \setminus \ann(b) $ (interchange the roles of $a$ and $b$ if necessary). Note that $ x \notin \m^2.$ This is because if $ x \in \m^2 =(l),$ then as $bl=0$ we get $x \in \ann(b),$ a contradiction to $x \notin \ann(b).$ We observe  the following.\vspace{0.1 cm} \newline
		(i) $ b\in \ann(a) \setminus \ann(x) \implies \ann(x) \neq \ann(a),$ i.e., $[x] \neq [a].$ \\
		(ii) $b \in \ann(b) \setminus  \ann(x) \implies \ann(x) \neq \ann(b),$ i.e., $[x] \neq [b].$ \\
		(iii)   $\bar x \neq 0 $ in $\m/\m^2 $ (since $ x \notin \m^2$). \vspace{0.1 cm} \newline
		We now show that $\bar x \in \langle \bar a \rangle .$  Suppose if $ \bar x \notin \langle  \bar a \rangle, $ then as $\dim \m/\m^2 =2,$ we see that $\{\bar x, \bar a\}$ forms a basis for $\m/\m^2.$ By Nakayama lemma, $\m = (x, a).$ But since $ ax= a^2 =0,$ we obtain $ a\in (0: \m) =\soc R.$ This contradicts our choice of $a \notin \soc R$ and hence $\bar x \in  \langle  \bar a \rangle .$ Thus $ x-ua =rl$ for some $u, r \in R.$ Multiplying both sides by $b,$ we get $bx= rbl.$ As noted above $bl=0,$ so  $bx=0.$ This is a contradiction to our choice of $ x \notin \ann(b).$ This contradiction shows that $\dim  \langle  \bar a, \bar b, \bar c \rangle =2,$ i.e., whenever we have a clique of size $3,$ say $ \{ [a], [b], [c]\},$ of the elements of $\m \setminus \m^2,$ then $ \langle \bar a, \bar b, \bar c \rangle  = \m/\m^2.$   We now claim that,\vspace{0.1 cm} \newline
		\underline{Claim 2}: $\omega(\Gamma_E(R)) \leq 4.$ \vspace{0.1 cm} \newline
		Suppose on the contrary that $\omega(\Gamma_E(R)) \geq  5,$ i.e., suppose there is a clique of size $5$ in $\Gamma_E(R).$ Discarding the two elements $[l]$ and $[k]$ coming from $\m^2$ and $\m^3$, if occurring in the clique of size $5,$ we are left with a clique of size $3,$ say $\{[a_1], [a_2], [a_3]\}$ such that $ a_i \in \m \setminus \m^2.$ By Claim 1,   we have $\dim \langle  \bar a_1, \bar a_2, \bar a_3 \rangle  = 2, $ i.e.,  $\m/\m^2 = \langle  \bar a_1, \bar a_2, \bar a_3 \rangle .$ So  (renaming if necessary) we have, $\m/\m^2 = \langle  \bar a_1, \bar a_2 \rangle .$  By Nakayama lemma we get $\m= (a_1, a_2).$ Now  $a_3 a_1=  a_3 a_2=0 \implies  a_3 \in 0 : \m =\soc R.$  Since $\soc R \subseteq \m^2,$ we get $a_3 \in \m^2.$ This is a contradiction to our choice of $a_3 \notin \m^2.$ Thus $\omega(\Gamma_E(R)) \leq 4.$\vspace{0.1 cm} \newline
		{\bf Case(ii):} We now consider the case when $\dim_k \soc R \geq 2.$ \vspace{0.1 cm} \newline We first  check that $\soc R \nsubseteq \m^2.$ Suppose on the contrary $\soc R \subseteq \m^2=(l).$ As noticed in the beginning of the proof, we have $ l \notin \soc R,$  so the inclusion $\soc R \subsetneqq \m^2$ is strict.   Note that \[ \ell(\m^3) =H_R(3) =1 < \ell(\soc R)= \dim_k(\soc R) = 2.\] Since $ (k)=\m^3 \subseteq \soc R, $ the inclusion $\m^3 \subsetneqq \soc R $ is strict. Thus we have $\m^3 \subsetneqq \soc R \subsetneqq \m^2.$  Hence $ \ell(\soc R/\m^3) \geq 1$ and $(\m^2/\soc R) \geq 1.$ Consider the following short exact sequence
		\[ 0 \rightarrow \soc R/\m^3 \rightarrow \m^2/\m^3 \rightarrow \m^2/\soc R \rightarrow 0. \]
		Now, computing the lengths from the short exact sequence,  we get $$ \ell(\m^2/\m^3) = \ell(\soc R/\m^3 ) + \ell(\m^2/\soc R ).$$ As noted above, $ \ell(\soc R/\m^3) \geq 1$ and $(\m^2/\soc R) \geq 1,$ it follows that  $ \ell(\m^2/\m^3) \geq 2.$ This is a contradiction to $\dim_k(\m^2/\m^3) =H_R(2)= 1.$ This shows that $\soc R \nsubseteq \m^2.$
		Hence there exists  $s \in \soc R $ such that  $ s\notin \m^2.$ Since $ s\notin \m^2$, $ \langle \bar s \rangle $ is a one dimensional subspace of  $\m/\m^2.$  We now claim, \vspace{0.1 cm} \newline
		\underline{Claim 3}: For  any $ x\in \m $ with $ \langle  \bar x \rangle  = \langle \bar s \rangle, $ we have either  $[x] = [k]$ or $[x]=[l].$ \vspace{0.1 cm} \newline To prove this claim, note that $ x - us = rl $ for some $ u, r \in R,$ where $u$ is a unit. We now have two possibilities either $ r \in \m$ or $ r \notin \m.$ If $ r \in \m,$ then $rl \in \m^3=(k)$ and so $ rl = ck$ for some $c \in R.$ Hence $x- us = ck.$ Since $ s, k \in \soc R,$ we get $ x \in \soc R$ and hence $ [x]=[k]$ in this case. If $ r \notin \m,$ then $r$ is a unit. Then in this case, we show that $ \ann(x) = \ann(l).$ For this, we consider $ p \in \ann(x)$ and multiply both sides of the equation $ x - us = rl $ by $p.$ As $ s \in \soc R$, $ps =0.$ This gives  us $prl =0.$ Now, as $r$ is a unit, we get $ pl=0.$ Hence $ p \in \ann(l)$ proving that $ \ann(x) \subseteq \ann(l).$ Similarly, we have $ \ann(l) \subseteq \ann(x).$ Therefore, when $ r \notin \m,$ we get $ \ann(x) =\ann(l),$ i.e., $[x] =[l].$
		This proves Claim 3.\vspace{0.1 cm} \newline 
		\underline{Claim 4}: $ \omega ( \Gamma_E(R)) \leq 2.$ \vspace{0.1 cm} \newline 
		We show that if $[a] \neq [b]$ in $\Gamma_E(R)$ are such that $[a] \neq [l], [k]$ and $[b] \neq [l], [k],$  then  $[a]$ can not be adjacent to $[b].$  Suppose on the contrary such $[a]$ and  $[b]$ are adjacent, i.e., $ab=0.$ Then by Claim 3, we obtain $ \langle  \bar a \rangle  \neq \langle  \bar s \rangle $ and   $\langle  \bar b \rangle  \neq \langle  \bar s \rangle .$  Hence, $\dim \langle  \bar a, \bar s \rangle =2.$ This gives us $\m/\m^2 =\langle  \bar a, \bar s \rangle.$ By  Nakayama lemma, we get $\m = (a, s).$ Now $ba =bs =0$ gives $ b \m =0,$ i.e., $b \in 0: \m = \soc R.$ Thus $[b] =[k]$  a contradiction to the fact that $[b] \neq [k].$ \\
		\indent  Now, suppose on the contrary that  $ \omega ( \Gamma_E(R)) \geq 3.$ By the observation above, the only possibility for a clique of size $3$ is $\{[l], [k], [c] \}$ for some $ c\in \m.$ Hence $lk=lc=kc=0.$  By Claim 3 again, we have $\langle  \bar c \rangle  \neq \langle  \bar s \rangle .$  Using Nakayama lemma as before we find that $\m= (c, s).$ As $ lc = ls=0,$ it follows that $ l \in \soc R.$ This contradicts the fact that $ l \notin \soc R$ mentioned in the beginning of the proof. Thus  $ \omega ( \Gamma_E(R)) \leq 2.$ \newline In both the cases (i) and (ii), we observe that $\omega(\Gamma_E(R)) \leq 5- \dim_k \soc R.$
	\end{proof}
	
	With this preparation, we are now ready to prove the main result of this paper. 
	
	\begin{theorem} \label{main} Let $R$ be a ring of length $5.$ Then $ \omega(\Gamma_E(R)) < \infty.$ Moreover, if in addition $(R, \m, k)$ a local ring then we have
		\[ \omega(\Gamma_E(R)) \leq 5- \dim_k \soc R. \]	
	\end{theorem}
	\begin{proof}
		Writing $R$ as a finite direct product of Artinian local rings, we have \[ R \cong \displaystyle \bigoplus_{i=1}^s R_i\] with $R_i$  Artinian local rings. Note that $ \ell(R) = \displaystyle \sum_{i=1}^{s} \ell(R_i)$	and  $\ell(R_i) \geq 1$ for all $i.$ Observe that if $ s \geq 2,$ then $\ell(R_i) \leq 4$ for all $i.$ In this case, it follows from \cite[Proposition 5.3]{Coy} and \cite [Proposition 5.8]{Coy} that $ \omega(\Gamma_E(R_i)) < \infty$ for all $i.$ Now by   \cite[Proposition 5.1]{Coy}, we  get $ \omega(\Gamma_E(R)) < \infty.$ If $s=1,$ then $R$ is Artinian local ring of length $5.$ We denote $R$ by $(R, \m, k).$ Then by Observation  \ref{obs}, we have the  following possibilities for the Hilbert functions of $R,$
		\begin{enumerate}
			\item [\rm (i)] $ H_R(i)= 1$ for $ 0 \leq i \leq 4;$ 	
			\item [\rm (ii)]  $ H_R(0)= 1, ~ H_R(1)=3, ~ H_R(2)=1;$
			\item [\rm (iii)] $ H_R(0)= 1,~  H_R(1)=2, ~ H_R(2)=2;$
			\item [\rm (iv)] $ H_R(0)= 1, ~ H_R(1)=2, ~ H_R(2)=1, ~ H_R(3)=1.$
		\end{enumerate}
		As noted in the Observation \ref{obs}, when $R$ has a Hilbert function  of case (i), then $R$ is a principal ideal ring.  So $\m =(a), \m^2=(b), \m^3=(c)$ and $\m^4=(0).$ This gives us $V(\Gamma_E(R)) = \{ [a], [b], [c] \},$ i.e., $V(\Gamma_E(R))$ is a finite graph. Hence, in this case, we have $ \omega(\Gamma_E(R)) \leq 3.$ In the  case (ii), it follows from  Proposition \ref{bound} that $ \omega(\Gamma_E(R)) \leq 5- \dim_k \soc R.$ The conclusion in the case (iii) follows from Proposition \ref{bound1}, while that in case (iv) follows from Proposition \ref{bound2}. Thus $ \omega(\Gamma_E(R)) \leq 5- \dim_k \soc R $ in all the four cases (i) to (iv). This proves the result. 	 
	\end{proof}
	
	\section{\bf Acknowledgments}
	It is a pleasure to thank  Vinayak Joshi for many useful discussions on the subject of this paper. The author would like to thank DST-SERB for financial assistance under the project ECR/2017/00790. 
	
	\bibliographystyle{amsplain}
	\bibliography{ref}
\end{document}